\documentclass{article}

\usepackage{amsmath,amssymb,amsthm,amsrefs}
\usepackage{bbm}
\usepackage{enumerate}
\usepackage{pifont}
\usepackage{setspace}
\usepackage{tikz-cd} 
\usepackage[utf8]{inputenc}

\newcommand{\Z}{\mathbb{Z}}
\newcommand{\D}{\mathbb{D}}
 
\newcommand{\N}{\mathbb{N}} 
\newcommand{\R}{\mathbb{R}} 
\newcommand{\T}{\mathbb{T}} 
\newcommand{\C}{\mathbb{C}}
\newcommand{\1}{\mathbbm{1}}

\newcommand{\AC}{\mathcal{A}}

\newcommand{\DC}{\mathcal{D}}

\newcommand{\HC}{\mathcal{H}}
\newcommand{\IC}{\mathcal{I}}

\newcommand{\PC}{\mathcal{P}}

\newcommand{\SC}{\mathcal{S}}

\newcommand{\XC}{\mathcal{X}}

\newcommand{\rk}{\textnormal{rk}}

\usepackage[textwidth=136mm]{geometry}
\newcommand{\overbar}[1]{\mkern 1.0mu\overline{\mkern-1.0mu#1\mkern-1.0mu}\mkern 1.0mu}
\newcommand*\conj[1]{\overbar{#1}}

\newtheorem{theorem}{Theorem}[section]
\newtheorem{lemma}[theorem]{Lemma}
\newtheorem{definition}[theorem]{Definition}
\newtheorem{proposition}[theorem]{Proposition}
\newtheorem{corollary}[theorem]{Corollary}

\newtheorem{remark}[theorem]{Remark}

\title{On the characterization of Triebel--Lizorkin type spaces of analytic functions}
\author{Eskil Rydhe}

\begin{document}
\maketitle

\begin{abstract}
	We consider different characterizations of Triebel--Lizorkin type spaces of analytic functions on the unit disc. Even though our results appear in the folklore, detailed descriptions are hard to find, and in fact we are unable to discuss the full range of parameters. Without additional effort we work with vector-valued analytic functions, and also consider a generalized scale of function spaces, including for example so-called $Q$-spaces. The primary aim of this note is to generalize, and clarify, a remarkable result by Cohn and Verbitsky, on factorization of Triebel--Lizorkin spaces. Their result remains valid for functions taking values in an arbitrary Banach space, provided that the vector-valuedness ``sits in the right factor''. On the other hand, if we impose vector-valuedness on the ``wrong'' factor, then the factorization fails even for separable Hilbert spaces. 
\end{abstract}

\section{Introduction}

\begin{definition}
	Let $X$, $X_1$ and $X_2$ be normed linear spaces of analytic functions on $\D$. If for any $f\in X$ there exists $f_1\in X_1$ and $f_2\in X_2$ such that $f=f_1f_2$ and
	\begin{equation*}
	\sup_{f\in X\setminus\left\{0\right\}}\inf_{f_1f_2=f}\frac{ \|f_1\|_{X_1} \|f_2\|_{X_2}}{ \|f\|_X}<\infty,
	\end{equation*}
	then we say that $X\subset X_1\cdot X_2$. If for any $f_1\in X_1$ and $f_2\in X_2$ it holds that $f_1f_2\in X$ and
	\begin{equation*}
	\sup_{f_1\in X_1\setminus\left\{0\right\},f_2\in X_2\setminus\left\{0\right\}}\frac{ \|f_1f_2\|_X}{ \|f_1\|_{X_1} \|f_2\|_{X_2}}<\infty,
	\end{equation*}
	then we say that $X_1\cdot X_2\subset X$. If $X\subset X_1\cdot X_2$ and  $X_1\cdot X_2\subset X$, then we say that  $X=X_1\cdot X_2$.
\end{definition}

Throughout this paper, we let $\XC$ and $\HC$ respectively denote a general Banach space and a separable Hilbert space, both complex. By $\AC\left(\XC\right)$ we denote the space of analytic $\XC$-valued functions on the open unit disc $\D$. For short, we write $\AC=\AC\left(\C\right)$. The same principle will apply to all function spaces discussed below. 

We let $\T$ denote be the unit circle in $\C$, and give it the parametrization $x\mapsto \zeta_x$, where $\zeta_x=e^{2\pi i x}$, $x\in\R$. For $p\in\left(0,\infty\right)$, we denote by $L^p\left(\T,\XC\right)$ the class of strongly measurable functions $f:\T\to \XC$ such that $ \|f\|_{L^p\left(\T,\XC\right)}^p=\int_{\zeta_x\in\T} \|f\left(\zeta_x\right)\|_\XC^p\, dx <\infty$, where we somewhat abusively write $dx$ to indicate Lebesgue integration with respect to $\zeta_x$. We will often identify $f\in L^1\left(\T,\XC\right)$ with its Poisson extension $\PC\left[f\right]:\D\to\XC$. Under this identification, the Fourier coefficients $\hat f\left(n\right) = \int_{\T}f\left(\zeta_x\right)\conj{\zeta_x^n}\, dx$ are the Taylor coefficients of $\PC\left[f\right]$. For short, we typically write $f$ in place of $\PC\left[f\right]$. We denote the $n$th Taylor coefficient of a general function $f\in\AC\left(\XC\right)$ by $\hat f\left(n\right)$, even though $f$ is not necessarily the Poisson extension of an integrable function.

We define the Hardy space $H^p\left(\XC\right)$ as the class of functions $f\in\AC\left(\XC\right)$ such that $ \|f\|_{H^p\left(\XC\right)}=\sup_{0<r<1} \|f_r\|_{L^p\left(\T,\XC\right)} <\infty$, where $f_r:w\mapsto f\left(rw\right)$. In the case where $\XC=\HC$, we have the so-called square function characterization of $H^p\left(\HC\right)$; in the language of Section \ref{Sec:Fspaces}, $H^p\left(\HC\right)=F_{p,2}^0\left(\HC\right)$. If $p\ge 1$, then $H^p\left(\HC\right)$ also coincides with the space of $f\in L^p\left(\T,\HC\right)$ such that $\hat f\left(n\right)=0$ for $n<0$.

We define the pairing $\left\langle f,\phi\right\rangle_{\AC\left(\HC\right)} = \sum\langle \hat f\left(n\right),\hat \phi\left(n\right)\rangle_\HC$, where $f,\phi\in\AC\left(\HC\right)$, and $f$ is polynomial. If $\phi\in H^1\left(\HC\right)$, then $\left\langle f,\phi\right\rangle_{\AC\left(\HC\right)} = \int_{\T} \left\langle f, \phi \right\rangle_\HC\, dx$. With respect to this pairing, the dual of $H^1\left(\HC\right)$ is given by $BMOA\left(\HC\right)$, the space of $\phi \in H^1\left(\HC\right)$ such that
\begin{equation*}
 \|\phi\|_{BMOA\left(\HC\right)}= \|\phi\left(0\right)\|_\HC+\sup_{\textnormal{arcs }I\subset\T}\frac{1}{|I|}\int_I \|\phi\left(\zeta_x\right)-\frac{1}{|I|}\int_I\phi\left(\zeta_y\right)dy\|_\HC\, dx<\infty.
\end{equation*}
A characterization of $BMOA\left(\HC\right)$ relevant to this paper is, in the language of Section \ref{Sec:Fspaces}, that $BMOA\left(\HC\right)=F_{\infty,2}^0\left(\HC\right)$. This is the so-called Carleson measure characterization of $BMOA\left(\HC\right)$.

Given $\alpha\in\R$ and $f\in\AC\left(\XC\right)$, we define the fractional derivative $D^\alpha f$ by
\begin{equation*}
D^\alpha f\left(w\right)=\sum_{n=0}^\infty \left(1+n\right)^\alpha\hat f\left(n\right)w^n,\quad w\in\D.
\end{equation*}
Consider the class $D^\alpha H^p:=\left\{f\in\AC;D^{-\alpha}f\in H^p\right\}$ equipped with the norm $ \|f\|_{D^\alpha H^p} := \|D^{-\alpha}f\|_{H^p}$. The following result is due to Cohn and Verbitsky \cite{Cohn-Verbitsky2000:FactTentSpacesHankOps}*{Theorem 2}:

\begin{theorem}\label{thm:Cohn-Verbitsky}
	Let $\alpha>0$, $0<p,p_1,p_2<\infty$, and $p_1^{-1}+p_2^{-1}=p^{-1}$. Then
	\begin{equation*}
	D^\alpha H^p=H^{p_1}\cdot D^\alpha H^{p_2}.
	\end{equation*}
\end{theorem}

The present author's interest in the above result arose while studying the following type of bilinear forms appearing naturally in control theory, e.g. \citelist{\cite{Janson-Peetre1988:Paracomms}\cite{Rydhe2016:VecHankOpsCarlesonEmbsBMOA}}: Given $\phi\in\AC$ and $\alpha>0$, we define the bilinear Hankel type form
\begin{equation}\label{Eq:HankFormScal}
H_{\phi,\alpha}:H^2\times H^2\ni \left(g,h\right)\mapsto \left\langle h,D^\alpha\left(\phi\conj{g}\right)\right\rangle_{\AC},
\end{equation}
on analytic polynomials. The next result on $H^2$-boundedness of $H_{\phi,\alpha}$ has several proofs in the literature, e.g. \citelist{\cite{Janson-Peetre1988:Paracomms}\cite{Rydhe2016:VecHankOpsCarlesonEmbsBMOA}}. As an illustration, we prove it by applying Theorem \ref{thm:Cohn-Verbitsky}:
\begin{proposition}\label{prop:Application}
	$H_{\phi,\alpha}$ is bounded if and only if $D^\alpha \phi \in BMOA$.
\end{proposition}
\begin{proof}
	Suppose that $g,h\in H^2$, and let $f\in H^1$ be a suitable function such that $D^\alpha f= g\left(D^\alpha h\right)$. Then
	\begin{equation*}
	\left\langle h,D^\alpha\left(\phi\conj{g}\right)\right\rangle_\AC=
	\left\langle D^\alpha h,\phi\conj{g}\right\rangle_\AC=
	\left\langle g\left(D^\alpha h\right) ,\phi\right\rangle_\AC=
	\left\langle D^\alpha f,\phi\right\rangle_\AC=
	\left\langle f,D^\alpha \phi\right\rangle_\AC.
	\end{equation*}
	The statement now follows from the Fefferman $H^1-BMOA$ duality theorem.
\end{proof}

The primary aim of this paper is to consider vector-valued generalizations of Theorem \ref{thm:Cohn-Verbitsky}. Given $\phi\in\AC\left(\HC\right)$ and $\alpha>0$, there are two natural analogues of \eqref{Eq:HankFormScal}:
\begin{equation}\label{Eq:HankForm}
H_{\phi,\alpha}:H^2\times H^2\left(\HC\right)\ni \left(g,h\right)\mapsto \left\langle h,D^\alpha\left(\phi\conj{g}\right)\right\rangle_{\AC\left(\HC\right)},
\end{equation}
and
\begin{equation}\label{Eq:HankForm*}
H_{\phi,\alpha}^*:H^2\left(\HC\right)\times H^2\ni \left(g,h\right)\mapsto \left\langle h,D^\alpha\left(\left\langle \phi,g\right\rangle _\HC\right)\right\rangle_{\AC}.
\end{equation}
The proof of Proposition \ref{prop:Application} now leads us to the following questions:
\begin{itemize}
	\item[Q1:] \begin{center} Is $D^\alpha H^1\left(\HC\right)= H^2 \cdot \left(D^\alpha H^2\left(\HC\right)\right)$?\end{center}
	\item[Q2:] \begin{center} Is $D^\alpha H^1\left(\HC\right)=H^2\left(\HC\right)\cdot\left(D^\alpha H^2\right)$ ?\end{center}
\end{itemize}
The first question will receive a positive answer. This yields that $H_{\phi,\alpha}$ is bounded if and only if $D^\alpha \phi\in H^1\left(\HC\right)^*=BMOA\left(\HC\right)$, a result also obtained in \cite{Rydhe2016:VecHankOpsCarlesonEmbsBMOA}. The second question receives a negative answer. Indeed, if the answer was positive, then $H_{\phi,\alpha}$ and $H_{\phi,\alpha}^*$ would be simultaneously bounded. This would contradict the following result, essentially due to Davidson and Paulsen \cite{Davidson-Paulsen1997:PolBddOps}. See also \cite{Rydhe2016:VecHankOpsCarlesonEmbsBMOA}*{Section 4}:
\begin{proposition}\label{prop:HankForms}
	Let $\alpha>0$, $\phi\in\AC\left(\HC\right)$, and define the forms $H_{\phi,\alpha}$ and $H_{\phi,\alpha}^*$ by \eqref{Eq:HankForm} and \eqref{Eq:HankForm*} respectively. If $H_{\phi,\alpha}$ is bounded, then $H_{\phi,\alpha}^*$ is also bounded. The converse does not hold.
\end{proposition}

As in \cite{Cohn-Verbitsky2000:FactTentSpacesHankOps}, we state our main result in the language of Triebel--Lizorkin spaces $F_{p,q}^s$. Their definition is quite elaborate, and we refer to Section \ref{Sec:Fspaces}, where we discuss different characterizations. Remarkably, our results hold for analytic functions taking values in an arbitrary Banach space:
\begin{theorem}\label{thm:FspaceFact}
	Let $\XC$ be a complex Banach space, $s<0$, $0<p<\infty$ and $1\le q<\infty$. Then
	\begin{equation*}
	F_{p,q}^s\left(\XC\right)=H^p \cdot F_{\infty,q}^s\left(\XC\right).
	\end{equation*}
	Moreover, the $H^p$-factor can be constructed as an outer analytic function.
\end{theorem}
\begin{remark}
	The attentive reader will of course note that we are lacking a statement for $0<q<1$. This is perhaps the main shortcoming of this paper, and it adheres to the fact that for this range of parameters we are not able to define the corresponding spaces.
\end{remark}

Since $H^p=H^{p_1}\cdot H^{p_2}$ whenever $p_1^{-1}+p_2^{-1}=p^{-1}$, Theorem \ref{thm:FspaceFact} implies that 
\begin{equation*}
F_{p,q}^s\left(\XC\right)=H^p \cdot F_{\infty,q}^s\left(\XC\right)=H^{p_1} \cdot H^{p_2} \cdot F_{\infty,q}^s\left(\XC\right)=H^{p_1} \cdot F_{p_2,q}^s\left(\XC\right).
\end{equation*}
Theorem \ref{thm:Cohn-Verbitsky} corresponds to the special case $D^\alpha H^p=F_{p,2}^{-\alpha}$. For emphasis, we state a corollary:
\begin{corollary}\label{thm:FspaceFact2}
	Let $\XC$ be a complex Banach space, $s<0$, $0<p,p_1,p_2<\infty$ and $1\le q<\infty$. If $p_1^{-1}+p_2^{-1}=p^{-1}$, then
	\begin{equation*}
	F_{p,q}^s\left(\XC\right)=H^{p_1} \cdot F_{p_2,q}^s\left(\XC\right).
	\end{equation*}
\end{corollary}

We also obtain a non-factorization result:
\begin{theorem}\label{thm:FspaceProd}
	Let $\XC$ be a complex Banach space, $s<0$, $0<p<\infty$ and $1\le q<\infty$. Then
	\begin{equation*}
	H^p\left(\XC\right) \cdot F_{\infty,q}^s\subset F_{p,q}^s\left(\XC\right).
	\end{equation*}
	In general, this inclusion is strict. In particular, Proposition \ref{prop:HankForms} shows that if $\HC$ is an infinite-dimensional Hilbert space, then there exists $f\in F_{1,2}^s\left(\HC\right)$, such that for any $g\in H^1\left(\HC\right)$ and $h\in F_{\infty,2}^s$, $f\ne gh$.
\end{theorem}

The first ingredient needed in order to generalize Theorem \ref{thm:Cohn-Verbitsky} is a factorization result for tent spaces, Theorem \ref{thm:TspaceFact}. We point out that the proofs from \cite{Cohn-Verbitsky2000:FactTentSpacesHankOps} go through also in the vector-valued case (replacing moduli with vector space norms). However, the scalar-valued result even implies the vector-valued one. We demonstrate how in Section \ref{Sec:Tentspaces}.

We will also need some properties of Triebel--Lizorkin spaces of analytic functions on $\D$. These appear in the literature (e.g. \cite{Cohn-Verbitsky2000:FactTentSpacesHankOps}) but I have not been able to find any stringent justification. The vector-valued setting that we consider does not require any additional effort. Nevertheless, this setting does not appear to have been considered before. For these reasons we dedicate Section \ref{Sec:Fspaces} to establishing some rudimentary theory. We develop our theory in the language of the more general class of distribution spaces $F_{p,q}^{s,\tau}\left(\R^d\right)$ introduced in by Yang and Yuan \citelist{\cite{Yang-Yuan2008:NewFcnSpacesTriebel--LizorkinSpacesQSpaces}\cite{Yang-Yuan2010:NewBesov-TypeTriebel-Lizorkin-TypeSpacesQSpaces}}. Rather than increasing our efforts, $F_{p,q}^{s,\tau}\left(\XC\right)$ unifies the spaces $F_{p,q}^s\left(\XC\right)$ where $p<\infty$, and $F_{\infty,q}^s\left(\XC\right)$, perhaps even decreasing the amount of work needed. Another motivation to study the generalized scale is that it encompasses more spaces, for example the so-called $Q$-spaces introduced by Aulaskari, Xiao and Zhao \cite{Aulaskari-Xiao-Zhao1995:SubspacesSubsetsBMOAUBC}.

\section{Preliminaries and notation}\label{Sec:Prelim}
We use the standard notation $\Z$, $\R$, and $\C$ for the respective rings of integers, real numbers, and complex numbers. In addition, $\D  =\left\{w\in\C;|w|<1\right\}$ and $\T  =\left\{\zeta\in\C;|\zeta|=1\right\}$. We will often identify $\T$ with $\R/\Z$, using the map $x\mapsto e^{2\pi i x}$. Subsets of $\R/\Z$ and $\T$ are identified in a similar way. In particular, we let the set of dyadic arcs $\DC\left(\T\right)$ be the image of the set $\DC\left(\left[0,1\right)\right)=\left\{\left[2^{-j}k,2^{-j}\left(k+1\right)\right);j\in\N_0,0\le k \le 2^j-1\right\}$. Note that in general, an arc $I\subset \T$ may correspond to the union of two intervals in $\left[0,1\right)$. We use the letters $x,y,z$ to denote generic points on $\R$. By $\zeta_x$ we denote the point $e^{2\pi i x}\in\T$. The arc-wise distance between $\zeta_x$ and $\zeta_y$ is denoted by $|\zeta_x-\zeta_y|$. A Euclidean ball with radius $r$ and center $w$ is denoted $B\left(w,r\right)$.  Given two parametrized sets of nonnegative numbers $\left\{A_i\right\}_{i\in I}$ and $\left\{B_i\right\}_{i\in I}$, we use the notation $A_i\lesssim B_i$, $i\in I$ to indicate the existence of a positive constant $C$ such that $A_i\le CB_i$ whenever $i\in I$. Sometimes we allow ourselves to not mention the index set $I$ and instead let it be implicit from the context. If $A_i\lesssim B_i$ and $B_i\lesssim A_i$, then we write $A_i\approx B_i$.

For a background on the Bochner--Lebesgue classes $L^p\left(\T,\XC\right)$ we refer to \cite{Diestel-Uhl1977:VecMeasures}. Given a strongly measurable function $f:\T\to\XC$, we define the corresponding Hardy-Littlewood maximal function by 
\begin{equation*}
Mf\left(\zeta_x\right)=\sup_{\substack{\textnormal{arcs }I\subset \T ;\\ I \textnormal{ centered at }\zeta_x}}\int_{\zeta_y\in I} \|f\left(\zeta_y\right)\|_\XC\,  dy,\quad \zeta_x\in\T.
\end{equation*}
The following periodic analogue of the vector-valued maximal theorem follows easily from \cite{Fefferman-Stein1971:MaximalIneq}*{Theorem 1}:
\begin{theorem}\label{thm:Fefferman-Stein}
	Let $\beta,\gamma\in\left(1,\infty\right)$. Then there is a number $K=K\left(\beta,\gamma\right)$ such that
	\begin{equation*}
	 \|\left(\sum_{n=1}^\infty |Mf_n|^\beta\right)^{1/\beta}\|_{L^\gamma\left(\T\right)}
	\le
	K_{\beta,\gamma}  \|\left(\sum_{n=1}^\infty |f_n|^\beta\right)^{1/\beta}\|_{L^\gamma\left(\T\right)}
	\end{equation*}
	whenever $\left(f_n\right)_{n=1}^\infty$ is a sequence of measurable $\C$-valued functions on $\T$. 
\end{theorem}

For $f\in L^1\left(\T,\XC\right)$, we define the Poisson extension
\begin{equation*}
f_r\left(\zeta_x\right)=\PC\left[f\right]\left(w\right)=\int_\T f\left(\zeta_y\right) P_r\left(\zeta_{x-y}\right)\, dy,\quad w=r\zeta_x,
\end{equation*}
where
\begin{equation*}
P_r\left(\zeta_y\right)=\frac{1-r^2}{1-2r\cos\left(2\pi y\right)+r^2}=\frac{1}{1-r\zeta_y}+\frac{r\conj{\zeta_y}}{1-r\conj{\zeta_y}},\quad y\in\R,
\end{equation*}
is the Poisson kernel for $\D$. By geometric summation, $\hat P_r\left(n\right)=r^{|n|}$. It is well-known that $ \|f_r\|_{L^p\left(\T,\XC\right)}\le  \|f\|_{L^p\left(\T,\XC\right)}$, and that $ \|f_r\left(x\right)\|_\XC \lesssim Mf\left(x\right)$, e.g. \cite{Garnett2007:BddAnalFcnsBook}*{Sections I.3 an I.4}. Note also that
\begin{equation}\label{Eq:PoissonDecay}
P_r\left(\zeta_y\right)\approx \frac{1}{1-r}\frac{1}{1+\left(\frac{|\zeta_y-1|}{1-r}\right)^2},\quad \zeta_y\in \T.
\end{equation}

A function $v:\D\to \left[-\infty,\infty\right)$ is called upper semi-continuous if for each $w_0\in\D$, $\limsup_{w\to w_0}v\left(w\right)\le v\left(w_0\right)$. If $v$ is upper semi-continuous and if for each $w_0\in\D$ there exists $r_0>0$ such that for each $0<r<r_0$, $B\left(w_0,r\right)\subset \D$ and
\begin{equation*}
v\left(w_0\right)\le \frac{1}{\pi r^2}\int_{B\left(w_0,r\right)}v\left(w\right)\, dA\left(w\right),
\end{equation*} 
then we say that $v$ is subharmonic. If $v$ is subharmonic, then the function $\left[0,1\right)\ni r\mapsto \int_\T v\left(r\zeta_x\right)\, dx$ is increasing. If $v$ is subharmonic and extends continuously to $\T$, then $v$ is majorized by the Poisson extension of its boundary values, i.e.
\begin{equation*}
v\left(w\right)\le \int_\T v\left(\zeta_y\right)P_r\left(\zeta_{x-y}\right)\, dy,\quad w=r\zeta_x,
\end{equation*}
For proofs of these claims, we refer to \cite{Garnett2007:BddAnalFcnsBook}. If $f\in\AC\left(\XC\right)$, then for any $0<p<\infty$, the function $\D\ni w\mapsto  \|f\left(w\right)\|_\XC^p$ is subharmonic, e.g. \cite{Rosenblum-Rovnyak1985:HardyClassesOpTheory}*{Chapter 4}.

A Stolz angle $\Gamma\left(\zeta_x\right)$ is the convex hull of the set $\left\{\zeta_x\right\}\cup\frac{1}{2}\T$. The non-tangential maximal function of $f\in\AC\left(\XC\right)$ is given by
\begin{align*}
A_\infty f\left(\zeta_x\right)=\sup_{w\in\Gamma\left(\zeta_x\right)} \|f\left(w\right)\|_\XC,\quad \zeta_x\in\T.
\end{align*}
It is well-known that $f\in H^p$ if and only if $f\in\AC$ and $A_\infty f\in L^p$. This so-called non-tangential maximal characterization of $H^p$ carries over to the general $\XC$-valued setting. This is known, e.g. \cite{Blasco1988:HardySpacesVecValDuality}*{Lemma 1.1}, but for the reader's convenience we provide a short indication of a proof based on the Szegö--Solomentsev theorem \cite{Rosenblum-Rovnyak1985:HardyClassesOpTheory}*{Appendix 2}. It is obvious that $\left\{f\in\AC\left(\XC\right);\, A_\infty f\in L^p \right\}\subset H^p\left(\XC\right)$. On the other hand, if $f\in H^p\left(\XC\right)$, then $u\left(\zeta_x\right)=\lim_{r\to 1} \|f_r\left(\zeta_x\right)\|_\XC$ exists Lebesgue-a.e. By Fatou's lemma, $ \|u\|_{L^p\left(\T\right)}\le  \|f\|_{H^p\left(\XC\right)}$. Moreover, $\log u\in L^1\left(\T\right)$ and $\log \|f\left(w\right)\|_\XC\le h\left(w\right)$, where $h=\PC\left[\log u\right]$. Let $\tilde h$ be the harmonic conjugate of $h$, with $\tilde h\left(0\right)=0$, and define $g=\exp(h+i\tilde h)$. Then $ \|f\|_\XC\le |g|$, and so $ \|f\|_{H^p\left(\XC\right)}\le  \|g\|_{H^p}$. A standard application of Jensen's inequality shows that $|g|^p\le \PC\left[u^p\right]$. Thus $ \|g\|_{H^p}\le  \|u\|_{L^p\left(\T\right)}\le  \|f\|_{H^p\left(\XC\right)}$. The non-tangential maximal characterization of $H^p\left(\XC\right)$ now follows from the scalar case, since
\begin{equation*}
 \|A_\infty f\|_{L^p}\le  \|A_\infty g\|_{L^p}\lesssim  \|g\|_{H^p} =  \|f\|_{H^p\left(\XC\right)}. 
\end{equation*}

The square function of $f\in\AC\left(\XC\right)$ is given by
\begin{align*}
S f\left(\zeta_x\right)=\int_{w\in\Gamma\left(\zeta_x\right)} \|\left(Df\right)\left(w\right)\|_\XC^2 \, dA\left(w\right),\quad \zeta_x\in\T.
\end{align*}
It is a famous result by Fefferman and Stein \cite{Fefferman-Stein1972:HpSpaces} that if $f\in\AC\left(\HC\right)$, then $A_\infty f\in L^p\left(\T\right)$ if and only if $Sf\in L^p\left(\T\right)$. In general, $H^p\left(\XC\right)$ may fail to have a square function characterization, e.g. \cite{Davidson-Paulsen1997:PolBddOps}*{Remark 4.11}. The duality between $H^1\left(\HC\right)$ and $BMOA\left(\HC\right)$ is a celebrated theorem by Fefferman \cite{Fefferman1971:CharBMO}, adapted to analytic functions on $\D$ (e.g. \cite{Garnett2007:BddAnalFcnsBook}*{Exercise VI.5}), with values in $\HC$ (e.g. \cite{Blasco1997:VecValBMOAGeomBSpaces}).

If $f\in H^p\left(\HC\right)$, then there exists $bf\in L^p\left(\T,\HC\right)$ such that $bf\left(\zeta_x\right)$ equals the non-tangential limit $\lim_{w\to\zeta_x}f\left(w\right)$ for almost every $\zeta_x$, and $f_r\to bf$ in $L^p\left(\T,\HC\right)$. Moreover, $f=\PC\left[bf\right]$, e.g. \cite{Rosenblum-Rovnyak1985:HardyClassesOpTheory}*{Chapter 4}. For this reason we will typically not distinguish between a function $f\in H^p\left(\HC\right)$, and its boundary values $bf\in L^p\left(\HC\right)$.

The convolution of $f\in\AC$ and $g\in \AC\left(\XC\right)$ is defined as $f\ast g\in \AC\left(\XC\right)$ with $\left(f\ast g\right)^{\hat{}}\left(n\right)=\hat f \left(n\right)\hat g\left(n\right)$. If $f$ and $g$ are Poisson extensions of integrable functions, then so is $f\ast g$, and
\begin{equation*}
\left(f\ast g\right) \left(\zeta_x\right)=\int_\T f\left(\zeta_{x-y}\right)g\left(\zeta_y\right)\, dy,\quad \zeta_x\in\T.
\end{equation*}

Given a smooth function $\varphi:\R\to\C$, let $\varphi^{\left(k\right)}$ denote its (classical) derivative of order $k\in\N_0$. The Schwartz space $\SC$ is the class of functions $\varphi:\R\to\C$ for which all derivatives decay faster than any rational function, i.e. $\varphi\in\SC$ if and only if for any $k,N\in\N_0$, $\sup_{x\in \R}\left(1+|x|\right)^N|\varphi^{\left(k\right)}\left(x\right)|<\infty$. The Fourier transform of $\varphi\in\SC$ is given by
\begin{equation*}
\hat \varphi\left(\xi\right)=\int_{\R} \varphi\left(x\right)e^{-2\pi i x \xi}\, dx,\quad \xi\in\R.
\end{equation*}
We will be interested in the $1$-periodization of $\varphi$ given by
\begin{equation*}
\Phi\left(\zeta_x\right)=\sum_{k\in\Z}\varphi\left(x-k\right),\quad x\in\R.
\end{equation*}
It is easy to see that if $N\ge 2$, then
\begin{equation}\label{Eq:PerDecay}
\sup_{\zeta_x\in\T}\left(1+|\zeta_x-1|\right)^N|\Phi\left(\zeta_x\right)|\lesssim \sup_{x\in \R}\left(1+|x|\right)^N|\varphi^{\left(k\right)}\left(x\right)|.
\end{equation}
Moreover $\hat \Phi \left(n\right)=\hat \varphi\left(n\right)$, $n\in\Z$.

\section{Tent spaces}\label{Sec:Tentspaces}
Given a subset $E\subset \T$, we define the ``tent'' over $E$ as
\begin{equation*}
T\left(E\right)=\left(\bigcup_{\zeta_x\notin E}\Gamma\left(\zeta_x\right)\right)^c.
\end{equation*}
Let $f:\D\to\C$ be a measurable function. For $q\in\left(0,\infty\right)$, we define the functional $A_q$ by
\begin{align*}
A_qf\left(x\right)=\left(\int_{\Gamma\left(x\right)}\frac{|f\left(z\right)|^q}{\left(1-|z|\right)^2}\, dA\left(z\right)\right)^{1/q},\quad x\in\T.
\end{align*}
Note that the functional $A_\infty$ was defined in the previous section. For $p,q\in\left(0,\infty\right)$, we say that $f\in T_{p,q}$ if $ \| f \|_{T_{p,q}}= \| A_{q}f \|_{L^p\left(\T\right)}<\infty $. For $q\in\left(0,\infty\right)$, we say that $f\in T_{\infty,q}$ if
\begin{equation*}
 \|f\|_{T_{\infty,q}}^q=\sup_{I\in\DC}\frac{1}{|I|}\int_{T\left(I\right)}\frac{|f\left(z\right)|^q}{\left(1-|z|\right)}\, dA\left(z\right)<\infty,
\end{equation*}
i.e. if $\frac{|f\left(z\right)|^q}{\left(1-|z|\right)}\, dA\left(z\right)$ is a Carleson measure.

For $p,q$ in the range discussed above, we define $T_{p,q}\left(\XC\right)$ to be the set of functions $f:\D\to \XC$ such that $ \|f\|_{\XC}\in T_{p,q}$. We equip this space with the obvious metric structure.

The main result of \cite{Cohn-Verbitsky2000:FactTentSpacesHankOps} is that $T_{p,q}=H^p\cdot T_{\infty,q}$. This result easily carries over to the $\XC$-valued setting, provided that we take care which one of the factors is $\XC$-valued:

\begin{theorem}\label{thm:TspaceFact}
	Let $\XC$ be a complex Banach space, and $0<p,q<\infty$. Then
	\begin{equation*}
	T_{p,q}\left(\XC\right) = H^p\cdot T_{\infty,q}\left(\XC\right),
	\end{equation*}
	and
	\begin{equation*}
	H^p\left(\XC\right)\cdot T_{\infty,q} \subset T_{p,q}\left(\XC\right).
	\end{equation*}
	If $\XC=\HC$ is an infinite-dimensional Hilbert space, then the inclusion is strict.
\end{theorem}

\begin{proof}
	If $f\in T_{p,q}\left(\XC\right)$, then by the scalar-valued result $ \|f\|_\XC=gH$, where $g\in H^p$ is outer analytic, $H\in T_{\infty,q}$, and $ \|g\|_{H^p} \|H\|_{T_{\infty,q}}\lesssim  \|f\|_{T_{p,q}\left(\XC\right)}$. Define $h=\frac{f}{g}$. Then $ \|h\|_\XC = |H| \in T_{\infty,q}$. By definition $h\in T_{\infty,q} \left(\XC\right)$, and $f=gh$. This proves that $T_{p,q}\left(\XC\right) \subset H^p\cdot T_{\infty,q}\left(\XC\right)$. The reverse inclusion $H^p\cdot T_{\infty,q}\left(\XC\right) \subset T_{p,q}\left(\XC\right)$ also follows from the scalar-valued result: Let $g\in H^p$ and $h\in T_{\infty,q}\left(\XC\right) $. Then
	\begin{equation*}	
	 \| gh \|_{ T_{p,q}\left(\XC \right)} 	=
	 \| g\|h \|_\XC \|_{ T_{p,q}} 	\lesssim  \| g \|_{H^p}  \|\, \|h \|_\XC\|_{ T_{\infty,q}} =  \| g \|_{H^p}  \| h\|_{ T_{\infty,q}\left(\XC\right)}. 
	\end{equation*}
	The statement that $H^p\left(\XC\right)\cdot T_{\infty,q} \subset T_{p,q}\left(\XC\right)$ follows similarly. For infinite-dimensional Hilbert spaces, this inclusion must be strict in order to not contradict Theorem \ref{thm:FspaceProd}.
\end{proof}

\section{Triebel-Lizorkin type spaces}\label{Sec:Fspaces}
The so called Triebel--Lizorkin spaces $F^s_{p,q}\left(\R^d\right)$, $s\in\R$, $0<p,q\le\infty$, are well-studied objects. An extensive treatment is given in Triebel's monographs \citelist{\cite{Triebel1983:TheoryOfFcnSpacesI}\cite{Triebel1992:TheoryOfFcnSpacesII}\cite{Triebel2006:TheoryOfFcnSpacesIII}}. We also mention papers by Liang, Sawano, Ullrich, Yang and Yuan \cite{Liang-Sawano-Ullrich-Yang-Yuan2012:CharBesov--Triebel--Lizorkin--HausdorffSpacesCoorbitsWavelets}, Peetre \cite{Peetre1975:OnSpacesOfTriebel--Lizorkin}, Rychkov \cite{Rychkov1999:OnThmOfBuiPaluszynskiTaibleson}\footnote{This paper contains a mistake, which is corrected in \cite{Ullrich2012:ContCharBesov--Triebel--LizorkinSpacesInterpretCoorbits}.}, and Ullrich \cite{Ullrich2012:ContCharBesov--Triebel--LizorkinSpacesInterpretCoorbits}, which give a more direct introduction to many of the ideas to be used in this paper. In this section we investigate Triebel--Lizorkin type spaces of $\XC$-valued analytic functions on $\D$. The more involved proofs are postponed to Subsection \ref{Ssec:pf}.

\begin{definition}\label{def:Fpqs}
	Let $0<p\le \infty$, $1\le q<\infty$, $s\in\R$ and $\alpha>s$. We define the Triebel--Lizorkin space $F_{p,q}^s\left(\D,\XC\right)$ as the space of functions $f\in\AC\left(\XC\right)$ such that
	\begin{equation*}
	 \|f\|_{F_{p,q}^s\left(\D,\XC\right)}:= \|\left(1-|z|\right)^{\alpha-s}D^\alpha f\|_{T_{p,q}\left(\XC\right)}<\infty.
	\end{equation*}
\end{definition}
The claim of this section is that the above definition is unambiguous, i.e. that it does not depend on the parameter $\alpha$. Moreover, the respective topologies defined for different choices of $\alpha$ are equivalent. If we for the moment accept this as a fact, then the proof of Theorem \ref{thm:FspaceFact} is indeed short:

\begin{proof}[Proof of Theorem \ref{thm:FspaceFact}]
	If $s<0$, then $f\in F_{p,q}^s\left(\D,\XC\right)$ if and only if $\left(1-|z|\right)^{-s}f\in T_{p,q}\left(\XC\right)$. By Theorem \ref{thm:TspaceFact}, this happens if and only if $f=gh$, where $g\in H^p$ is an outer function and $h=\left(1-|z|\right)^sH$, with $H\in T_{\infty,q}\left(\XC\right)$. Since $g$ is outer, $h\in\AC\left(\XC\right)$. Furthermore, $h\in F_{\infty,q}^s\left(\D,\XC\right)$ by definition.
\end{proof}

The remainder of this section is dedicated to the justification of Definition \ref{def:Fpqs}. In \citelist{\cite{Yang-Yuan2008:NewFcnSpacesTriebel--LizorkinSpacesQSpaces}\cite{Yang-Yuan2010:NewBesov-TypeTriebel-Lizorkin-TypeSpacesQSpaces}}, Yang and Yuan introduced the spaces $F^{s,\tau}_{p,q}\left(\R^d\right)$, $s\in\R$, $\tau\ge 0$, $p\in\left(0,\infty\right)$, $q\in\left(0,\infty\right]$. These include the standard Triebel--Lizorkin spaces: If $0<p<\infty$, then $F^{s,0}_{p,q}\left(\R^d\right)=F^{s}_{p,q}\left(\R^d\right)$, while $F^{s,1/p}_{p,q}\left(\R^d\right)=F^{s}_{\infty,q}\left(\R^d\right)$. In contrast to $F^{s}_{p,q}\left(\R^d\right)$, the spaces $F^{s,\tau}_{p,q}\left(\R^d\right)$ are not always distinct for different choices of parameters. On the other hand, they include for example the spaces $Q_\alpha\left(\R^d\right) = F^{\alpha,1/2-\alpha/d}_{2,2}\left(\R^d\right)$ introduced by Aulaskari, Xiao and Zhao \cite{Aulaskari-Xiao-Zhao1995:SubspacesSubsetsBMOAUBC}. This is in fact a motivation in \cite{Yang-Yuan2008:NewFcnSpacesTriebel--LizorkinSpacesQSpaces}. We chose to work with the more general scale of $F^{s,\tau}_{p,q}$-spaces, since this requires no additional effort.

A fundamental tool in the study of Triebel--Lizorkin spaces is the so-called Peetre maximal function: For $a>0$ and $f\in\AC\left(\XC\right)$, we define
\begin{equation}\label{Eq:Peetre}
f_{r,a}^*\left(\zeta_x\right)=\sup_{\zeta_y\in\T}\frac{ \|f_r\left(\zeta_y\right)\|_\XC}{\left(1+\frac{|\zeta_x-\zeta_y|}{1-r}\right)^a},\quad r \in\left[0,1\right),\, \zeta_x\in\T.
\end{equation}

\begin{definition}\label{def:Quasinorms}
	Let $\left(r_l\right)_{l\ge 0}$ be a sequence such that $0\le r_l<1$ and $2^l\left(1-r_l\right)\approx 1$. Furthermore, let $0<p,q<\infty$, $s\in \R$, $\tau\ge 0$ and $a>\max \left\{\frac{1}{p},\frac{1}{q}\right\}$. Define then following (quasi-)norms for $f\in\AC\left(\XC\right)$:
	\begin{align*}
	 \|f\|_1
	&=
	\sup_{I\in\DC\left(\T\right)}\frac{1}{|I|^\tau}\left(\int_{\zeta_x\in I}\left[\int_{\Gamma_I\left(\zeta_x\right)}\left(1-|w|\right)^{-2-sq} \|f\left(w\right)\|_\XC^q\, dA\left(w\right)\right]^{p/q} dx\right)^{1/p}.
	\\
	 \|f\|_2
	&=
	\sup_{I\in\DC\left(\T\right)}\frac{1}{|I|^\tau}\left(\int_{\zeta_x\in I}\left[\int_{r=1-|I|}^1\left(1-r\right)^{-1-sq} \|f_r\left(\zeta_x\right)\|_\XC^qdr\right]^{p/q} dx\right)^{1/p}.
	\\
	 \|f\|_3
	&=	
	\sup_{I\in\DC\left(\T\right)}\frac{1}{|I|^\tau}\left(\int_{\zeta_x\in I}\left[\sum_{l=\rk \left(I\right)}^\infty 2^{slq} \|f_l\left(\zeta_x\right)\|_\XC^q\right]^{p/q} dx\right)^{1/p}.
	\\
	 \|f\|_4
	&=
	\sup_{I\in\DC\left(\T\right)}\frac{1}{|I|^\tau}\left(\int_{\zeta_x\in I}\left[\sum_{l=\rk \left(I\right)}^\infty 2^{slq}f_{l,a}^*\left(\zeta_x\right)^q\right]^{p/q} dx\right)^{1/p}.
	\\
	 \|f\|_5
	&=
	\sup_{I\in\DC\left(\T\right)}\frac{1}{|I|^\tau}\left(\int_{\zeta_x\in I}\left[\int_{r=1-|I|}^1\left(1-r\right)^{-1-sq}f_{r,a}^*\left(\zeta_x\right)^q dr\right]^{p/q} dx\right)^{1/p}.
	\end{align*}
	If we wish to indicate the values of $p$, $q$, $s$ and $\tau$, then we use the notation $ \|f|_{p,q}^{s,\tau}\|_k$, $1\le k\le 5$. Similarly, $ \|f|\left(r_l\right)_{l\ge 0}\|_k$, $1\le k\le 5$, indicates the choice of $\left(r_l\right)_{l\ge 0}$.
\end{definition}

\begin{theorem}\label{thm:Norms}
	The (quasi-)norms in Definition \ref{def:Quasinorms} are comparable for $f\in\AC\left(\XC\right)$.
\end{theorem}

Let $\varphi\in\SC$ be a function such that for $\xi\ge 0$, $\hat \varphi\left(\xi\right)=e^{-\xi}$. With $\left(r_l\right)_{l\ge 0}$ as in definition \ref{def:Quasinorms}, let $t_l= \log\left(\frac{1}{r_l}\right)$, set $\varphi_l\left(x\right)=\frac{1}{t_l}\varphi\left(\frac{x}{r_l}\right)$, and let $\Phi_l$ denote the corresponding periodization. If $f\in\AC\left(\XC\right)$, then $\Phi_l\ast f=P_{r_l}\ast f$, and so
\begin{equation*}
 \|f\|_3
=	
\sup_{I\in\DC\left(\T\right)}\frac{1}{|I|^\tau}\left(\int_{\zeta_x\in I}\left[\sum_{l=\rk \left(I\right)}^\infty 2^{slq} \|\Phi_l\ast f\left(\zeta_x\right)\|_\XC^q\right]^{p/q}dx\right)^{1/p}.
\end{equation*}
This expression is a verbatim analogue of the defining (quasi-)norm for $F_{p,q}^{s,\tau}\left(\R^d\right)$. If $s<0$, then imposing finiteness of the above expression indeed gives us a space with the natural properties. However, for general $s\in\R$, such a definition would be severely flawed:
\begin{proposition}\label{prop:Flawed}
	With notation as in Definition \ref{def:Quasinorms}, and $s\ge 0$, if $f\in\AC\left(\XC\right)$ and $ \|f\|_3<\infty$, then $f\equiv 0$.
\end{proposition}
\begin{proof}
	Assume for simplicity that $p=q$. Then, by interchanging orders of integration,
	\begin{equation*}
	 \|f\|_3
	\ge	
	\left(\sum_{l=0}^\infty 2^{slq}\int_{\zeta_x\in \T} \|f_{r_l}\left(\zeta_x\right)\|_\XC^q\, dx\right)^{1/q}.
	\end{equation*}
	By subharmonicity, the right-hand side is infinite, unless $f\equiv 0$. The general case follows in the same way, with some simple modifications: If $p>q$, the we first apply Hölder's inequality to the integral:
	\begin{equation*}
	\int_{\zeta_x\in \T}\sum_{l=0}^\infty 2^{slq} \|f_{r_l}\left(\zeta_x\right)\|_\XC^q\, dx
	\le
	\left(\int_{\zeta_x\in \T}\left[\sum_{l=0}^\infty 2^{slq} \|f_{r_l}\left(\zeta_x\right)\|_\XC^q\right]^{p/q}\, dx\right)^{q/p}.
	\end{equation*}
	If $p<q$, then we instead use Hölder's inequality on a partial sum:
	\begin{equation*}
	\sum_{l=0}^N 2^{slp} \|f_{r_l}\left(\zeta_x\right)\|_\XC^p\le 
	\left(1+N\right)^{\frac{q-p}{q}}\left[\sum_{l=0}^N 2^{slq} \|f_{r_l}\left(\zeta_x\right)\|_\XC^q\right]^{p/q}.
	\end{equation*}
	Now integrate over $\T$, let $N\to \infty$, and argue by subharmonicity.
\end{proof}
A related observation is that if $s<0$ and $f\in\AC\left(\XC\right)$, then 
\begin{equation*}
\|f\left(w\right)\|_\XC\lesssim \|f\|_1\left(1-|w|\right)^{s+\tau-1/p},
\end{equation*}
cf. Lemma \ref{lemma:RadialGrowth}. So if $s+\tau-\frac{1}{p}>0$ and $ \|f\|_1<\infty$, then $f\equiv 0$ by the maximum principle for analytic functions. This motivates us to impose the restriction $\tau\le \frac{1}{p}$. The definition of $F_{p,q}^{s,\tau}\left(\D,\XC\right)$ is now inspired by the so-called lifting property $D^\alpha:F_{p,q}^s\left(\R^d\right)\to F_{p,q}^{s-\alpha}\left(\R^d\right)$.

\begin{definition}\label{def:Fpqst}
	Let $0<p,q<\infty$, $s\in\R$, $0\le\tau\le \frac{1}{p}$, and $\alpha>s$. We define the Triebel--Lizorkin type space $F_{p,q}^{s,\tau}\left(\D,\XC\right)$ as the space of functions $f\in\AC\left(\XC\right)$ such that
	\begin{equation*}
	 \|f\|_{F_{p,q}^{s,\tau}\left(\D,\XC\right)}:= \|D^\alpha f|_{p,q}^{s-\alpha,\tau}\|_1<\infty.
	\end{equation*}
\end{definition}

Note that if $0<p<\infty$, then $ \|f|_{p,q}^{s,0}\|_1 =  \|\left(1-|z|\right)^{-s} f\|_{T_{p,q}\left(\XC\right)}$ and $ \|f|_{q,q}^{s,1/q}\|_1\approx  \|\left(1-|z|\right)^{-s} f\|_{T_{\infty,q}\left(\XC\right)}$. Consequently $F_{p,q}^{s,0}\left(\D,\XC\right)=F_{p,q}^{s}\left(\D,\XC\right)$ and $F_{q,q}^{s,1/q}\left(\D,\XC\right)=F_{\infty,q}^{s}\left(\D,\XC\right)$. Definition \ref{def:Fpqst} is justified by the following lemma:
\begin{lemma}\label{lemma:HalfLifting} 
	Let $0<p,q<\infty$, $s\in\R$, $0\le\tau\le \frac{1}{p}$, and $\alpha>s$. Then	
	\begin{equation}\label{eq:HalfLifting}
	 \|D^\alpha f|_{p,q}^{s-\alpha,\tau}\|_3\lesssim  \|f|_{p,q}^{s,\tau}\|_4.
	\end{equation}
\end{lemma}

By Proposition \ref{prop:Flawed}, the above lemma is trivial if $s\ge 0$. Assume therefor that $s<0$, and let $\alpha>s$. In particular \eqref{eq:HalfLifting} holds. Moreover, $-\alpha >s-\alpha$. By another application of Lemma \ref{lemma:HalfLifting}, we obtain
\begin{equation*}
 \|f|_{p,q}^{s,\tau}\|_3= \|D^{-\alpha}\left(D^\alpha f\right)|_{p,q}^{s-\alpha -\left(-\alpha\right),\tau}\|_3\lesssim  \|D^\alpha f|_{p,q}^{s-\alpha ,\tau}\|_4.
\end{equation*}
Combined with Theorem \ref{thm:Norms}, this yields that
\begin{equation*}
 \|f|_{p,q}^{s,\tau}\|_1\lesssim  \|D^\alpha f|_{p,q}^{s-\alpha,\tau}\|_1\lesssim  \|f|_{p,q}^{s,\tau}\|_1,\quad f\in\AC\left(\XC\right),
\end{equation*}
whenever $s<0$ and $s-\alpha < 0$. This implies that if $\alpha_1,\alpha_2 > s$, then 
\begin{equation*}
 \|D^{\alpha_1} f|_{p,q}^{s-\alpha_1,\tau}\|_1\approx  \|D^{\alpha_2} f|_{p,q}^{s-\alpha_2,\tau}\|_1,
\end{equation*}
and thus $F_{p,q}^{s,\tau}\left(\D,\XC\right)$ is well-defined, with topology independent of $\alpha$.

\subsection{Proofs}\label{Ssec:pf}
In \ref{Sssec:Stability1} we quantify the rigidity of analytic functions in a certain way (Lemma \ref{lemma:Stability1}). We refer to this as ``the first stability property''. This will imply that $ \|f\|_3$ is essentially independent of $\left(r_l\right)_{l\ge 0}$, and also that $ \|f\|_1\lesssim  \|f\|_2\lesssim  \|f\|_3$. The proof that $ \|f\|_3\lesssim  \|f\|_1$ is simpler. In \ref{Sssec:Stability2}, we deduce a similar stability property (the second) for the Peetre maximal function (Lemma \ref{lemma:Stability2}). It follows that $ \|f\|_4\approx  \|f\|_5$. The estimate $ \|f\|_3\lesssim  \|f\|_4$ is trivial, since $f_l\left(\zeta_x\right)\le f_{l,a}^*\left(\zeta_x\right)$. We dedicate \ref{Sssec:ReverseMaximal} to obtaining the reverse maximal control, the most involved part of this paper. In \ref{Sssec:HalfLifting} we prove Lemma \ref{lemma:HalfLifting}.

\subsubsection{The first stability property}\label{Sssec:Stability1}
Given $I\in\DC\left(\T\right)$, we use the notation $I_n=I+n|I|$, for $1-\frac{1}{2|I|}\le n\le \frac{1}{2|I|}$. For other $n\in\Z$, we let $I_n$ be the empty set. Furthermore, we set $\IC_L=\cup_{|n|\le L}I_n$. 

\begin{lemma}\label{lemma:Stability1}
	Let $\XC$ be a Banach space, $\alpha\ge 0$ and $c_1,c_2,c_3,c_4>0$. Then there exists constants $K>0$ and $L\in\N$ with the following property:
	
	Let $I\in\DC\left(\T\right)$. If $r,r'\in\left[0,1\right)$ satisfy 
	\begin{equation*}
	\left(1-r\right)\le c_1\left(r'-r\right)\le c_2\left(1-\frac{r}{r'}\right)\le c_3\left(1-r\right)\le c_4 |I|,
	\end{equation*}
	i.e.
	\begin{equation*}
	\left(r'-r\right)\approx \left(1-\frac{r}{r'}\right)\approx \left(1-r\right)\lesssim |I|,
	\end{equation*}
	and $\zeta_x'\in\T$ satisfies $|\zeta_x-\zeta_x'|\le \alpha \left(1-r\right)$, where $\zeta_x\in I$, then for all $\delta >0$, for all $f\in\AC\left(\XC\right)$, it holds that
	\begin{equation}\label{Eq:Stability1}
	 \|f_r\left(\zeta_x'\right)\|_\XC^\delta \le K \left( M\left(\1_{\IC_L}  \|f_{r'}\|_\XC^\delta\right)\left(\zeta_x\right)+\sum_{|n|> L}\frac{1}{\left(|n|-L\right)^2|I|}\int_{I_n} \|f_{r'}\left(\zeta_y\right)\|_\XC^\delta \, dy\right).
	\end{equation}
\end{lemma}

\begin{proof}	
	By subharmonicity it holds that for any $L\in\N$
	\begin{align*}
	 \|f_{r}\left(\zeta_x'\right)\|_\XC^{\delta}\le {} & \int_\T P_{\frac{r}{r'}}\left(\zeta_{x'-y}\right) \|f_{r'}\left(\zeta_y\right)\|_\XC^\delta\, dy
	\\
	= {} &
	\int_{\IC_L} P_{\frac{r}{r'}}\left(\zeta_{x'-y}\right) \|f_{r'}\left(\zeta_y\right)\|_\XC^\delta\, dy	
	+
	\sum_{|n|> L}\int_{I_n} P_{\frac{r}{r'}}\left(\zeta_{x'-y}\right) \|f_{r'}\left(\zeta_y\right)\|_\XC^\delta\, dy.
	\end{align*}
	By the non-tangential maximal control of Poisson extensions,
	\begin{equation*}
	\int_{\IC_L} P_{\frac{r}{r'}}\left(\zeta_{x'-y}\right) \|f_{r'}\left(\zeta_y\right)\|_\XC^\delta\,  dy=\PC\left[\1_{\IC_L} \|f_{r'}\|_\XC^\delta\right]\left(\zeta_x'\right)\lesssim M\left(\1_{\IC_L} \|f_{r'}\|_\XC^\delta\right)\left(\zeta_x\right).
	\end{equation*}
	Provided that $L$ is sufficiently big, it follows from \eqref{Eq:PoissonDecay} that
	\begin{equation*}
	P_{\frac{r_l}{r_l'}}\left(\zeta_{x'-y}\right)\lesssim \frac{1}{\left(|n|-L\right)^2|I|},\quad \zeta_y\in I_n,|n|>L.
	\end{equation*}
\end{proof}

\begin{proof}[$ \|f\|_3$ is independent of $\left(r_l\right)_{l\ge 0}$:]
Let $\left(r_l\right)_{l\ge 0}$ and $\left(r_l'\right)_{l\ge 0}$ be sequences in $\left[0,1\right)$ such that
\begin{equation*}
2^{-l}c\le 1-r_l\le 2^{-l}C,\quad\textnormal{and}\quad 2^{-l}c'\le 1-r_l'\le 2^{-l}C'.
\end{equation*}
Chose $M\in\N$ such that $2^{M-1}c\ge C'$. For any $l\in\N_0$ it then holds that $r_l< r_{l+M}'$, and $2^{-1-l}c\le r_{l+M}'-r_l\le 2^{-l}C$. If we want to prove that $ \|f|\left(r_l\right)_{l\ge 0}\|_1\lesssim  \|f|\left(r_l'\right)_{l\ge 0}\|_1$, then by a shift of index we may assume that $M=0$. In this case our sequence has the asymptotic behavior
\begin{equation}\label{Eq:Asymptotic}
\left(r_l'-r_l\right)\approx \left(1-\frac{r_l}{r_l'}\right)\approx \left(1-r_l\right)\approx 2^{-l}.
\end{equation}
Let $\delta<\min\left\{p,q\right\}$ and $I\in\DC\left(\T\right)$. It follows from Lemma \ref{lemma:Stability1} (with $\alpha=0$) that
\begin{equation*}
 \|f_l\left(\zeta_x\right)\|_\XC^{\delta}\lesssim \underbrace{M\left(\1_{\IC_L} \|f_{l'}\|_\XC^\delta\right)\left(\zeta_x\right)}_{=:A_l}
+\sum_{|n|>L}\frac{1}{\left(|n|-L\right)^2}\underbrace{\frac{1}{|I|}\int_{I_n}  \|f_{l'}\left(\zeta_y\right)\|_\XC^\delta \, dy}_{=:B_{l,n}},
\end{equation*}
for $l\ge \rk\left(I\right)$ and $x\in I$. We now exploit the fact that $p/q=\delta/q\cdot p/\delta$. By Minkowski's inequality,
\begin{equation}\label{Eq:KeyEst1}
\begin{split}
\left(\int_I\left[\sum_{l=\rk \left(I\right)}^\infty 2^{slq} \|f_l\left(\zeta_x\right)\|_\XC^{q/\delta}\right]^{p/q}  dx\right)^{\delta/p}
\lesssim {} 
\left(\int_I\left[\sum_{l=\rk \left(I\right)}^\infty 2^{slq}A_l^{q/\delta}\right]^{p/q} dx\right)^{\delta/p}
\\
+
\sum_{|n|>L}\frac{1}{\left(|n|-L\right)^2}\left(\int_I\left[\sum_{l=\rk \left(I\right)}^\infty 2^{slq}B_{l,n}^{q/\delta}\right]^{p/q} dx\right)^{\delta/p}.
\end{split}
\end{equation}
By Theorem \ref{thm:Fefferman-Stein},
\begin{align*}
\left(\int_I\left[\sum_{l=\rk \left(I\right)}^\infty 2^{slq}A_l^{q/\delta}\right]^{p/q} dx\right)^{\delta /p}
&=
 \|\left(\sum_{l=\rk\left(I\right)}^\infty \left(M\left(2^{sl\delta}\1_{\IC_L}\| f_{l'} \|_\XC^\delta\right)\right)^{q/\delta}\right)^{\delta/q}\|_{L^{p/\delta}}
\\
&\lesssim
 \|\left(\sum_{l=\rk\left(I\right)}^\infty 2^{slq}\1_{\IC_L}\|f_{l'} \|_\XC^q\right)^{1/q}\|_{L^p}^\delta
\\
&\lesssim
\left(\sum_{|n|\le L} \int_{I_n} \left[\sum_{l=\rk \left(I\right)}^\infty 2^{slq} \|f_{l'}\left(\zeta_y\right)\|_\XC^q\right]^{p/q} dy\right)^{\delta/p}
\\
&\le
|I|^{\delta \tau} \|f|\left(r_l'\right)_{l\ge 0}\|_{3}^{\delta }.
\end{align*}
By Minkowski's and Jensen's inequalities
\begin{align*}
&\left(\int_I\left[\sum_{l=\rk \left(I\right)}^\infty 2^{slq}B_{l,n}^{q/\delta}\right]^{p/q}dx\right)^{\delta/p}
\\
&=
\left(\int_I\left(\left[\sum_{l=\rk \left(I\right)}^\infty \left(\frac{2^{sl\delta}}{|I|}\int_{I_n}  \|f_{r_l'}\left(\zeta_y\right)\|_\XC^\delta\, dy \right)^{q/\delta}\right]^{\delta/q}\right)^{p/\delta}dx\right)^{\delta/p}
\\
&\le
\left(\int_I\left(\frac{1}{|I|}\int_{I_n}\left[\sum_{l=\rk \left(I\right)}^\infty 2^{slq} \|f_{r_l'}\left(\zeta_y\right)\|_\XC^q \right]^{\delta/q}dy\right)^{p/\delta}dx\right)^{\delta/p}
\\
&\le
\left(\int_I\frac{1}{|I|}\int_{I_n}\left[\sum_{l=\rk \left(I\right)}^\infty 2^{slq} \|f_{r_l'}\left(\zeta_y\right)\|_\XC^q \right]^{p/q}dy\, dx\right)^{\delta/p}
\\
&\le |I|^{\delta \tau} \|f|\left(r_l'\right)_{l\ge 0}\|_3^\delta.
\end{align*}
Since $\sum_{|n|>L}\frac{1}{\left(|n|-L\right)^2}<\infty$, this proves that $ \|f|\left(r_l\right)_{l\ge 0}\|_\lesssim  \|f|\left(r_l'\right)_{l\ge 0}\|_3$.
\end{proof}

\begin{proof}[$ \|f\|_1\lesssim  \|f\|_2$:]
Let $I\in\DC\left(\T\right)$. Obviously,
\begin{multline*}
\int_{\Gamma_I\left(x\right)}\left(1-|w|\right)^{-2-sq} \|f\left(w\right)\|_\XC^q\, dA\left(w\right)
\\
\le
\int_{r=1-|I|}^1\left(1-r\right)^{-2-sq}\int_{|\zeta_x-\zeta_y|\le \theta_r} \|f_r\left(\zeta_y\right)\|_\XC^q\, dy\, dr,
\end{multline*}
where $\theta_r\approx 1-r$. Let $r'=\frac{1+r}{2}$. By Lemma \ref{lemma:Stability1},
\begin{equation*}
 \|f_r\left(\zeta_x'\right)\|_\XC^\delta \lesssim  M\left(\1_{\IC_L}  \|f_{r'}\|_\XC^\delta\right)\left(\zeta_x\right)+\sum_{|n|> L}\frac{1}{\left(|n|-L\right)^2}\frac{1}{|I|}\int_{I_n} \|f_{r'}\left(\zeta_y\right)\|_\XC^\delta\, dy,
\end{equation*}
for $|\zeta_x'-\zeta_x|\le \theta_r$. It follows that
\begin{align*}
\int_{|\zeta_x-\zeta_y|\le \theta_r} \|f_r\left(\zeta_y\right)\|_\XC^q\, dy
\lesssim
\left(1-r\right)\left[A_{r'}+\sum_{|n|> L}\frac{1}{\left(|n|-L\right)^2}B_{r',n}\right]^{q/\delta},
\end{align*}
where
\begin{equation*}
A_r:=M\left(\1_{\IC_L}  \|f_{r}\|_\XC^\delta\right)\left(\zeta_x\right),\quad\textnormal{and}\quad 
B_{r,n}:=\frac{1}{|I|}\int_{I_n} \|f_{r}\left(\zeta_y\right)\|_\XC^\delta\, dy.
\end{equation*}
By Minkowski's inequality, and the change of variables $r\mapsto r'$,
\begin{multline*}
%\begin{split}
\left(\int_I\left[\int_{\Gamma_I\left(\zeta_x\right)}\left(1-|w|\right)^{-2-sq} \|f\left(w\right)\|_\XC^q\, dA\left(w\right)\right]^{p/q} dx\right)^{\delta/p}
\\
\lesssim {} 
\left(\int_I\left[\int_{r=1-\frac{1}{2|I|}} \left(1-r\right)^{-1-sq} A_r^{q/\delta}dr\right]^{p/q} dx\right)^{\delta/p}
\\
+
\sum_{|n|>L}\frac{1}{\left(|n|-L\right)^2}\left(\int_I\left[\int_{r=1-\frac{|I|}{2}} \left(1-r\right)^{-1-sq} B_{r,n}^{q/\delta}dr\right]^{p/q} dx\right)^{\delta/p}.
%\end{split}
\end{multline*}
To complete the proof, we now proceed as from \eqref{Eq:KeyEst1}, with the obvious modifications.
\end{proof}

\begin{proof}[$ \|f\|_2\lesssim  \|f\|_3$:]
	It will prove convenient to work with $r_l=1-2^{-l}$. For any $l\in N_0$ and $r\in \left[r_{l},r_{l+1}\right]$ it holds that 
	\begin{equation*}
	\left(r_{l+2}-r\right)\approx \left(1-\frac{r}{r_{l+2}}\right)\approx \left(1-r\right)\approx 2^{-l}.
	\end{equation*}
	By Lemma \ref{lemma:Stability1}, it follows that
	\begin{equation*}
	 \|f_r\left(\zeta_x\right)\|_\XC^\delta \lesssim M\left(\1_{\IC_L} \|f_{r_{l+2}}\|_\XC^\delta\right)\left(\zeta_x\right) + \sum_{|n|> L}\frac{1}{\left(|n|-L\right)^2}\frac{1}{|I|}\int_{I_n} \|f_{r_{l+2}}\left(\zeta_y\right)\|_\XC^\delta\, dy,
	\end{equation*}
	for $I\in\DC\left(\T\right)$, $x\in I$, $l\ge \rk\left(I\right)$ and $r\in\left[r_{l},r_{l+1}\right]$. Since $\int_{r_{l+1}}^{r_{l+2}} \left(1-r\right)^{-1-sq}\, dr\approx 2^{slq}$, it follows that
	\begin{multline*}
	\left(\int_I\left[\int_{r=1-|I|}\left(1-r\right)^{-1-sq} \|f_r\left(\zeta_x\right)\|_\XC^q\, dr\right]^{p/q} dx\right)^{\delta/p}
	\\
	\lesssim {}
	\left(\int_I\left[\sum_{l=\rk\left(I\right)+2} 2^{slq} A_l^{q/\delta}\right]^{p/q} dx\right)^{\delta/p}
	\\
	+
	\sum_{|n|>L}\frac{1}{\left(|n|-L\right)^2}\left(\int_I\left[\sum_{l=\rk\left(I\right)+2}2^{slq} B_{l,n}^{q/\delta}\right]^{p/q} dx\right)^{\delta/p},
	\end{multline*}
	where
	\begin{equation*}
	A_l:=M\left(\1_{\IC_L} \|f_{r_{l+2}}\|_\XC^\delta\right)\left(\zeta_x\right),
	\quad\textnormal{and}\quad 
	B_{l,n}:=\frac{1}{|I|}\int_{I_n} \|f_{r_{l+2}}\left(\zeta_y\right)\|_\XC^\delta\, dy.
	\end{equation*}
	Once again, we proceed as from \eqref{Eq:KeyEst1}.
\end{proof}

\begin{proof}[$ \|f\|_3\lesssim  \|f\|_1$:]
	We work with the sequence given by $r_l=1-2^{-1-l}$. Given $x\in\IC\in\DC\left(\T\right)$, for each $l\ge\rk\left(I\right)$ there exists a ball $B_l=B\left(r_l\zeta_x,d_l\right)\subset\Gamma_I\left(x\right)$. Moreover, these balls may be chosen so that they are disjoint and $d_l\gtrsim 2^{-l}$. By subharmonicity,
	\begin{multline*}
	\sum_{l=\rk \left(I\right)}^\infty 2^{slq} \|f_l\left(\zeta_x\right)\|_\XC^q 
	\lesssim
	\sum_{l=\rk \left(I\right)}^\infty 2^{\left(2+sq\right)l}\int_{B_l} \|f\left(w\right)\|_\XC^q\, dA\left(w\right)
	\\
	\lesssim \int_{\Gamma_I\left(\zeta_x\right)}\left(1-|w|\right)^{-2-sq} \|f\left(w\right)\|_\XC^q\, dA\left(w\right).
	\end{multline*}
	The statement follows.
\end{proof}

\subsubsection{The second stability property}\label{Sssec:Stability2}
\begin{lemma}\label{lemma:Stability2}
	If $x\in \T$ and $\left(r_l\right)_{l\ge 0},\left(r_l'\right)_{l\ge 0}\subset\left[0,1\right)$ are sequences such that \eqref{Eq:Asymptotic} holds, then
	\begin{equation*}
	f^*_{l,a}\left(\zeta_x\right)\lesssim f^*_{l',a}\left(\zeta_x\right).
	\end{equation*}
\end{lemma}

\begin{proof}	
	The proof relies on an idea that will be useful several times in this section. Chose $\varphi\in\SC$ such that $\hat \varphi\left(\xi\right)=e^{-2\pi \xi}$, $\xi\ge -1$. For $t>0$, let $\varphi_t\left(x\right)=\frac{1}{t}\varphi\left(\frac{x}{t}\right)$, and $\Phi_r\left(\zeta_x\right)=\sum_{k\in\Z}\varphi_t\left(x-k\right)$, where $r=e^{-2\pi t}$. For $f\in\AC\left(\XC\right)$ it holds that  $f_r=P_r\ast f=\Phi_r\ast f$. Note that by \eqref{Eq:PerDecay}, for any integer $N\ge 2$,
	\begin{equation}\label{Eq:Decay}
	|\Phi_r\left(\zeta_x\right)|\lesssim \frac{1}{1-r}\frac{1}{\left(1+\frac{|\zeta_x-1|}{1-r}\right)^N},\quad r\in\left[0,1\right),\, \zeta_x\in \T.
	\end{equation}
	If we replaced $\Phi_r$ with the ordinary Poisson kernel, then the above inequality would hold only for $N=2$.
	
	Now note that
	\begin{equation*}
	f_{l}\left(\zeta_y\right)=\int_{\zeta_z\in\T}\Phi_\frac{r_l}{r_l'}\left(\zeta_{y-z}\right)f_{l'}\left(\zeta_z\right)\, dz.
	\end{equation*}
	By the triangle inequality, together with \eqref{Eq:Peetre} and the elementary inequality
	\begin{equation}\label{Eq:ElemIneq}
	1+b|\zeta_x-\zeta_y|\le \left(1+b|\zeta_x-\zeta_z|\right)\left(1+b|\zeta_z-\zeta_y|\right),\quad x,y,z\in\R,b>0,
	\end{equation}
	we obtain
	\begin{align*}
	\frac{ \|f_{l}\left(\zeta_y\right)\|_\XC}{\left(1+\frac{|\zeta_x-\zeta_y|}{1-r_l}\right)^a}
	&\lesssim	\int_{\zeta_z\in\T}\frac{|\Phi_\frac{r_l}{r_l'}\left(y-z\right)| \|f_{l'}\left(z\right)\|_\XC}{\left(1+2^l|\zeta_x-\zeta_y|\right)^a}\, dz
	\\
	&\lesssim
	f^*_{l',a}\left(\zeta_x\right)	\int_{\zeta_z\in\T}|\Phi_\frac{r_l}{r_l'}\left(\zeta_{y-z}\right)|\frac{\left(1+2^l|\zeta_x-\zeta_z|\right)^a}{\left(1+2^l|\zeta_x-\zeta_y|\right)^a}\, dz
	\\
	&\le
	f^*_{l',a}\left(\zeta_x\right)
	\int_{\zeta_z\in\T}|\Phi_\frac{r_l}{r_l'}\left(\zeta_{y-z}\right)|\left(1+2^l|\zeta_y-\zeta_z|\right)^a\, dz.
	\end{align*}
	By \eqref{Eq:Asymptotic} and \eqref{Eq:Decay}, we see that if $N>a+1$, then
	\begin{align*}
	\int_{\zeta_z\in\T}|\Phi_\frac{r_l}{r_l'}\left(\zeta_{y-z}\right)|\left(1+2^l|\zeta_y-\zeta_z|\right)^a\, dz\lesssim 1.
	\end{align*}
	The statement follows.
\end{proof}

\begin{proof}[$ \|f\|_4$ is independent of $\left(r_l\right)_{l\ge 0}$:]
	Once again, we may assume that \eqref{Eq:Asymptotic} holds. The statement follows immediately from Lemma \ref{lemma:Stability2}.
\end{proof}

\begin{proof}[$ \|f\|_4\approx  \|f\|_5$:]
		We may chose $\left(r_l\right)_{l\ge 0}$ so that $r_l=1-2^{-l}$. Note that
	\begin{equation*}
	\int_{r=1-|I|}^1\left(1-r\right)^{-1-sq}f^*_{r,a}\left(\zeta_x\right)^q\, dr=\sum_{l=\rk\left(I\right)}^\infty \int_{r=r_l}^{r_{l+1}}\left(1-r\right)^{-1-sq}f^*_{r,a}\left(\zeta_x\right)^q\, dr.
	\end{equation*}
	For an arbitrary sequence $\left(r_l'\right)_{l\ge 0}$ such that $r_{l+2}'\in\left[r_{l},r_{l+1}\right]$, \eqref{Eq:Asymptotic} holds. By an application of Lemma \ref{lemma:Stability2}, we obtain
	\begin{align*}
	\int_{r=r_l}^{r_{l+1}}\left(1-r\right)^{-1-sq}f^*_{r,a}\left(\zeta_x\right)^q\, dr &\lesssim \int_{r=r_l}^{r_{l+1}}\left(1-r\right)^{-1-sq}f^*_{l+2,a}\left(\zeta_x\right)^q\, dr 
	\\
	&\lesssim 
	2^{slq}f^*_{l+2,a}\left(\zeta_x\right).
	\end{align*}
	It follows that $ \|f\|_5\lesssim  \|f\|_4$. The reverse estimate is similar.
\end{proof}

\subsubsection{Reverse maximal control}\label{Sssec:ReverseMaximal}

It will be convenient to work with the sequence given by $r_l=e^{-2\pi 2^{-l}}$. Let $\varphi_t$ and $\Phi_r$ be as in the proof of Lemma \ref{lemma:Stability2}, and set $\varphi_m=\varphi_{2^{-m}}$ and $\Phi_m=\Phi_{r_m}$. Then $\Phi_m$ is the $1$-periodization of $\varphi_m$. Now choose $\left\{W_m\right\}_{m=1}^\infty\subset \SC$ such that $\textnormal{supp }\hat W_1\subset\left[-\frac{1}{2},2\right]$, $\textnormal{supp }\hat W_2\subset\left[1,4\right]$, $\hat W_m\left(\xi\right)=\hat W_{m-1}\left(\xi/2\right)$ for $m\ge 3$, and $\sum_{m=1}^\infty \hat W_m\left(\xi\right)=1$ for $\xi\ge 0$. Define $\left\{\psi_m\right\}_{m=1}^\infty\subset \SC$ by
\begin{equation*}
\hat \psi_m\left(\xi\right)=\frac{\hat W_m\left(\xi\right)}{\hat\varphi_m\left(\xi\right)},\quad m\in\N,\xi\in\R.
\end{equation*}
Then
\begin{equation*}
\sum_{m=1}^\infty \hat\psi_m\left(\xi\right)\hat\varphi_m\left(\xi\right)=1\qquad \textnormal{for }\xi\ge 0.
\end{equation*}
Furthermore, for any $l\in\N_0$ we have that
\begin{equation*}
\sum_{m=1}^\infty \hat\psi_m\left(2^{-l}\xi\right)\hat\varphi_m\left(2^{-l}\xi\right)
=
\sum_{m=1}^\infty \hat\psi_m\left(2^{-l} \xi\right)\hat\varphi_{m+l}\left(\xi\right)
=1,\quad \xi\ge 0.
\end{equation*}
Define $\lambda_{m,l}\in\SC$ by $\hat \lambda_{m,l}\left(\xi\right)=\hat\psi_m\left(2^{-l}\xi\right)$, $\xi\in\R$, and let $\Lambda_{m,l}$ denote the corresponding $1$-periodization. We thus obtain
\begin{equation}\label{Eq:Calderon}
\sum_{m=1}^\infty \hat\Lambda_{m,l}\left(n\right) \hat\Phi_{m+l}\left(n\right)
=
1,\quad l,n\in\N_0.
\end{equation}
This is a so-called Calderon reproducing type formula for analytic functions.

We will need the following technical lemma:
\begin{lemma}\label{lemma:CalderonDecay}
	If $N\in\N$, then there exists a $K>0$ such that
	\begin{align*}
	\sup_{\zeta_y\in\T} \left(1+2^l|\zeta_y-1|\right)^N |\Phi_{l}\ast\Lambda_{m,l}\left(\zeta_y\right)|\le K 2^l2^{-mN}
	\end{align*}
	for all $m\in\N$ and $l\in\N_0$.
\end{lemma}
\begin{proof}
	By \eqref{Eq:PerDecay}, it suffices to show that
	\begin{align*}
	\sup_{y\in\R} |y^N\left(\varphi_{l}\ast\lambda_{m,l}\right)\left(y\right)|\le K 2^l2^{-\left(l+m\right)N}
	\end{align*}
	For notational simplicity, we assume that $m\ge 2$. By elementary properties of the Fourier transform,
	\begin{align*}
	y^N\left(\varphi_{l}\ast\lambda_{m,l}\right)\left(y\right)
	&=
	\frac{1}{\left(2\pi i\right)^N}\int \left[\left(\frac{d}{d\xi}\right)^N\left(\hat\varphi_{l}\left(\xi\right)\hat\lambda_{m,l}\left(\xi\right)\right)\right]e^{2\pi i\xi y}\, d\xi
	\end{align*}
	Since $\hat\varphi_{l}\left(\xi\right)\hat\lambda_{m,l}\left(\xi\right)=e^{-2\pi \left(2^{-l}-2^{-l-m}\right)\xi} \hat W_2\left(2^{2-l-m}\xi\right)$, we may use Leibniz's rule together with the support of $\hat W_2$ to obtain
	\begin{align*}
	|y^N\left(\varphi_{l}\ast\lambda_{m,l}\right)\left(y\right)|
	&\lesssim
	\sum_{k=0}^N2^{-l\left(N-1\right)}2^{-m\left(N-k-1\right)}e^{-2\pi \left(2^m-1\right)}.
	\end{align*}
	This proves the statement for $m\ge 2$, since for any $k\in\N_0$,
	\begin{equation*}
	2^{-m\left(N-k-1\right)}e^{-2\pi \left(2^m-1\right)}\lesssim 2^{-mN}. 
	\end{equation*}
	The case $m=1$ is similar.
\end{proof}

The proof that $ \|\cdot\|_2\lesssim  \|\cdot\|_1$ is based on the following lemma:
\begin{lemma}\label{Maximalcontrol}
	Let $N\in\N$, $0<a\le N$, $\delta>0$ and $r_l=e^{-2\pi 2^l}$. Then there exists $K=K\left(a,N,\delta\right)>0$ such that
	\begin{equation*}
	f_{l,a}^*\left(\zeta_x\right)^\delta
	\le K
	\sum_{m=1}^\infty 2^{-mN\delta} 2^{m+l}\int_{y\in\T}\frac{ \|f_{l+m}\left(\zeta_y\right)\|_\XC^\delta}{\left(1+2^l|\zeta_x-\zeta_y|\right)^{a\delta}}\, dy
	\end{equation*}
	for $f\in\AC\left(\XC\right)$, $l\in\N_0$ and $x\in\T$.
\end{lemma}

\begin{proof}
	By \eqref{Eq:Calderon},
	\begin{align*}
	f_l\left(\zeta_x\right)&=\Phi_l\ast f\left(\zeta_x\right)
	\\
	&=
	\sum_{m=1}^\infty \Phi_l\ast \Lambda_{m,l}\ast\Phi_{l+m}\ast f\left(\zeta_x\right)
	\\
	&=
	\int_{\T} \left(\Phi_l\ast \Lambda_{m,l}\right)\left(\zeta_{x-y}\right)\left(\Phi_{l+m}\ast f\right)\left(\zeta_y\right)\, dy\\
	&=
	\int_{\T} \left(\Phi_l\ast \Lambda_{m,l}\right)\left(\zeta_{x-y}\right)\left(1+2^l|\zeta_x-\zeta_y|\right)^N\frac{\left(\Phi_{l+m}\ast f\right)\left(\zeta_y\right)}{\left(1+2^l|\zeta_x-\zeta_y|\right)^N}\, dy.
	\end{align*}
	By the triangle inequality and Lemma \ref{lemma:CalderonDecay}, we have that
	\begin{equation}\label{Eq:Prelim}
	\begin{split}
	 \|f_l\left(\zeta_x\right)\|_\XC 
	&\le
	\sum_{m=1}^\infty \sup_{\zeta_y\in\T} \left|\Phi_{r_l^\rho}\ast\Lambda_{m,l,\rho}\left(\zeta_y\right)\left(1+2^l|\zeta_y-1|\right)^N\right|
	\int_{\zeta_y\in\T}\frac{ \|f_{l+m}\left(\zeta_y\right)\|_\XC}{\left(1+2^l|\zeta_x-\zeta_y|\right)^N}\, dy
	\\
	&\lesssim
	\sum_{m=1}^\infty 2^{-mN}2^l
	\int_{\zeta_y\in\T}\frac{ \|f_{l+m}\left(\zeta_y\right)\|_\XC}{\left(1+2^l|\zeta_x-\zeta_y|\right)^N}\, dy.
	\end{split}
	\end{equation}
	If $\delta>1$, then we proceed as follows: Since $N$ is arbitrary, clearly
	\begin{align*}
	 \|f_l\left(\zeta_x\right)\|_\XC \lesssim
	\sum_{m=1}^\infty 2^{-m\left(N+1\right)}2^l
	\int_{\zeta_y\in\T}\frac{ \|f_{l+m}\left(\zeta_y\right)\|_\XC}{\left(1+2^l|\zeta_x-\zeta_y|\right)^{N+1}}\, dy.
	\end{align*}
	Applying Hölder's inequality twice we obtain
	\begin{align*}
	 \|f_l\left(\zeta_x\right)\|_\XC^\delta 
	&\lesssim
	\sum_{m=1}^\infty 2^{-mN\delta}2^{l\delta}
	\left(\int_{\zeta_y\in\T}\frac{ \|f_{l+m}\left(\zeta_y\right)\|_\XC}{\left(1+2^l|\zeta_x-\zeta_y|\right)^{N+1}}\, dy\right)^\delta
	\\
	&\lesssim
	\sum_{m=1}^\infty 2^{-mN\delta}2^{l}
	\int_{\zeta_y\in\T}\frac{ \|f_{l+m}\left(\zeta_y\right)\|_\XC^\delta}{\left(1+2^l|\zeta_x-\zeta_y|\right)^{N\delta}}\, dy.
	\end{align*}
	Now use that $a\le N$, divide by $\left(1+2^l|\zeta_z-\zeta_x|\right)^{a\delta}$ and use \eqref{Eq:ElemIneq} with $b=2^l$ get that
	\begin{align*}
	\frac{ \|f_l\left(\zeta_x\right)\|_\XC^\delta}{\left(1+2^l|\zeta_z-\zeta_x|\right)^{a\delta}} 
	\lesssim
	\sum_{m=1}^\infty 2^{-mN\delta}2^{l}
	\int_{\zeta_y\in\T}\frac{ \|f_{l+m}\left(\zeta_y\right)\|_\XC^\delta}{\left(1+2^l|\zeta_z-\zeta_y|\right)^{a\delta}}\, dy.
	\end{align*}
	This completes the proof for $\delta >1$.
	
	If $\delta\le 1$, then we instead do the following: By a shift of the index $l$ in \eqref{Eq:Prelim} we see that for each $k\in\N_0$,
	\begin{align*}
	2^{-kN} \|f_{k+l}\left(\zeta_x\right)\|_\XC 
	&\lesssim 
	\sum_{m=1}^\infty 2^{-\left(m+k\right)N}2^{l+k}
	\int_{\zeta_y\in\T}\frac{ \|f_{k+l+m}\left(\zeta_y\right)\|_\XC}{\left(1+2^{l+k}|\zeta_x-\zeta_y|\right)^N}\, dy
	\\
	&\le
	\sum_{m=1}^\infty 2^{-\left(m+k\right)N}2^{m+l+k}
	\int_{\zeta_y\in\T}\frac{ \|f_{k+l+m}\left(\zeta_y\right)\|_\XC}{\left(1+2^{l}|\zeta_x-\zeta_y|\right)^N}\, dy
	\\
	&=
	\sum_{m=1+k}^\infty 2^{-mN}2^{m+l}
	\int_{\zeta_y\in\T}\frac{ \|f_{l+m}\left(\zeta_y\right)\|_\XC}{\left(1+2^{l}|\zeta_x-\zeta_y|\right)^N}\, dy
	\\
	&\le
	\sum_{m=1}^\infty 2^{-mN}2^{m+l}
	\int_{\zeta_y\in\T}\frac{ \|f_{l+m}\left(\zeta_y\right)\|_\XC}{\left(1+2^{l}|\zeta_x-\zeta_y|\right)^N}\, dy.
	\end{align*}
	Using \eqref{Eq:ElemIneq} again, we have that
	\begin{align}\label{Eq:Prelim2}
	\frac{2^{-kN} \|f_{l+k}\left(\zeta_x\right)\|_\XC}{\left(1+2^l|\zeta_z-\zeta_x|\right)^a} 
	\lesssim 
	\sum_{m=1}^\infty 2^{-mN}2^{m+l}
	\int_{\zeta_y\in\T}\frac{ \|f_{m+l}\left(\zeta_y\right)\|_\XC}{\left(1+2^{l}|\zeta_z-\zeta_y|\right)^a}\, dy.
	\end{align}
	Introduce the maximal function
	\begin{align*}
	M_lf\left(\zeta_x\right)=\sup_{k\in\N_0}\sup_{\zeta_y\in\T}\frac{2^{-kN} \|f_{l+k}\left(\zeta_y\right)\|_\XC}{\left(1+2^l|\zeta_x-\zeta_y|\right)^a},\quad \zeta_x\in\T.
	\end{align*}
	By \eqref{Eq:Prelim2},
	\begin{align*}
	M_lf\left(\zeta_x\right)
	\lesssim 
	\sum_{m=1}^\infty 2^{-mN\delta }2^{m+l}M_lf\left(\zeta_x\right)^{1-\delta}
	\int_{\zeta_y\in\T}\frac{ \|f_{m+l}\left(\zeta_y\right)\|_\XC^\delta}{\left(1+2^{l}|\zeta_z-\zeta_y|\right)^{a\delta}}\, dy.
	\end{align*}
	By Lemma \ref{lemma2} below, this implies that
	\begin{align*}
	M_{l,\rho}\left(\zeta_x\right)^\delta\lesssim 
	\sum_{m=1}^\infty 2^{-mN\delta}2^{m+l}
	\int_{\zeta_y\in\T}\frac{ \|f_{m+l}\left(\zeta_y\right)\|_\XC^\delta}{\left(1+2^{l}|\zeta_z-\zeta_y|\right)^{a\delta}}\, dy,
	\end{align*}
	whenever the right-hand side i finite. This completes the proof.
\end{proof}

\begin{lemma}\label{lemma2}
	If 
	\begin{align*}
	\sum_{m=1}^\infty 2^{-mN\delta} 2^{m+l}\int_{y\in\T}\frac{ \|f_{m+l}\left(\zeta_y\right)\|_\XC^\delta}{\left(1+2^l|\zeta_x-\zeta_y|\right)^{a\delta}}\, dy <\infty,
	\end{align*}
	then $M_lf\left(\zeta_x\right)<\infty$.
\end{lemma}

\begin{proof}
	Since $l$ is a fixed number and $\T$ is compact we may equivalently show that
	if
	\begin{align}
	\sum_{m=1}^\infty 2^{m\left(1-N\delta\right)}\int_{\zeta_y\in\T}  \|f_{m+l}\left(\zeta_y\right)\|_\XC^\delta\, dy <\infty, \label{converges}
	\end{align}
	then
	\begin{align*}
	\sup_{m\in\N_0}\sup_{\zeta_y\in\T} 2^{-mN\delta} \|f_{m+l}\left(\zeta_y\right)\|_\XC^\delta<\infty.
	\end{align*}
	By subharmonicity
	\begin{align*}
	 \|f_{m+l}\left(\zeta_y\right)\|_\XC^\delta
	\le
	\int  \|f_{m+l+1}\|_\XC^\delta P_{\frac{r_{l+m}}{r_{l+m+1}}}\left(\zeta_y\right)\, dy
	\lesssim
	2^{m+l} \int  \|f_{m+l+1}\left(\zeta_y\right)\|_\XC^\delta\, dy
	\lesssim
	2^l2^{mN\delta},
	\end{align*}
	where the last inequality holds whenever \eqref{converges} converges.
\end{proof}

We now prove that $ \|\cdot\|_2\lesssim  \|\cdot\|_1$. Fix $\delta\in \left(\frac{1}{a},\min\left\{p,q\right\}\right)$, and chose $N\in\N$ such that $N>a$ and $1-N\delta -sq<0$. Given $I\in\DC\left(\T\right)$, define $I_n=I+n|I|$, where $1-\frac{1}{2|I|}\le n\le \frac{1}{2|I|}$, and $3I=\cup_{|n|\le 1}I_n$. By Lemma \ref{Maximalcontrol},
\begin{align*}
f_{l,a}^*\left(\zeta_x\right)^\delta
\lesssim {} &
\sum_{m=1}^\infty 2^{-mN\delta} 2^{m+l}\int_{\T}\frac{ \|f_{l+m}\left(\zeta_y\right)\|_\XC^\delta}{\left(1+2^l|\zeta_x-\zeta_y|\right)^{a\delta}}\, dy.
\end{align*}
We treat the right-hand side as in the proof of Lemma \ref{lemma:Stability1}.
\begin{align*}
f_{l,a}^*\left(\zeta_x\right)^\delta
\lesssim {} &
\sum_{m=1}^\infty 2^{-mN\delta} 2^{m+l}\int_{3I}\frac{ \|f_{l+m}\left(\zeta_y\right)\|_\XC^\delta}{\left(1+2^l|\zeta_x-\zeta_y|\right)^{a\delta}}\, dy
\\
&+
\sum_{|n|\ge 2}\sum_{m=1}^\infty 2^{-mN\delta} 2^{m+l}\int_{I_n}\frac{ \|f_{l+m}\left(\zeta_y\right)\|_\XC^\delta}{\left(1+2^l|\zeta_x-\zeta_y|\right)^{a\delta}}\, dy
\\
\lesssim {} &
\underbrace{\sum_{m=1}^\infty 2^{-mN\delta} 2^{m}M\left(\1_{3I} \|f_{m+l}\|_\XC^\delta\right)\left(x\right)}_{=:A_l}
\\
&+
\sum_{|n|\ge 2}\underbrace{\sum_{m=1}^\infty 2^{-mN\delta} 2^{m+l}\frac{1}{2^{la\delta}|n|^{a\delta}|I|^{a\delta}}\int_{I_n} \|f_{l+m}\left(\zeta_y\right)\|_\XC^\delta \, dy}_{=:B_{l,n}}.
\end{align*}
By successive applications of Minkowski's inequality,
\begin{align*}
\left(\int_{\zeta_x\in I}\left[\sum_{l=\rk \left(I\right)}^\infty 2^{slq}f_{l,a}^*\left(\zeta_x\right)^q\right]^{p/q}dx\right)^{\delta/p}
\lesssim {} &
\left(\int_{x\in I}\left[\sum_{l=\rk \left(I\right)}^\infty 2^{slq}A_l^{q/\delta}\right]^{p/q}dx\right)^{\delta/p}
\\
&+
\sum_{|n|\ge 2}\left(\int_{\zeta_x\in I}\left[\sum_{l=\rk \left(I\right)}^\infty 2^{slq}B_{l,n}^{q/\delta}\right]^{p/q}dx\right)^{\delta/p}.
\end{align*}
By Jensen's inequality and rearrangement of terms,
\begin{align*}
\sum_{l=\rk \left(I\right)}^\infty 2^{slq}A_l^{q/\delta}
&\lesssim
\sum_{l=\rk \left(I\right)}^\infty 2^{slq}\sum_{m=1}^\infty 2^{m\left(1-N\delta\right)} M\left(\1_{3I} \|f_{m+l}\|_\XC^\delta\right)\left(\zeta_x\right)^{q/\delta}
\\
&=
\sum_{l=\rk \left(I\right)+1}^\infty 2^{slq}\sum_{m=1}^{l-\rk\left(I\right)} 2^{m\left(1-N\delta-sq\right)} M\left(\1_{3I} \|f_{l}\|_\XC^\delta\right)\left(\zeta_x\right)^{q/\delta}
\\
&\lesssim
\sum_{l=\rk \left(I\right)+1}^\infty 2^{slq} M\left(\1_{3I} \|f_{l}\|_\XC^\delta\right)\left(\zeta_x\right)^{q/\delta}.
\end{align*}
It follows from Theorem \ref{thm:Fefferman-Stein} that
\begin{equation*}
\left(\int_{x\in I}\left[\sum_{l=\rk \left(I\right)}^\infty 2^{slq}A_l^{q/\delta}\right]^{p/q}dx\right)^{\delta/p}\lesssim  \|f\|_3|I|^{\tau\delta}.
\end{equation*}
The corresponding estimate for $B_{l,n}$,
\begin{equation*}
\sum_{|n|\ge 2}\left(\int_{\zeta_x\in I}\left[\sum_{l=\rk \left(I\right)}^\infty 2^{slq}B_{l,n}^{q/\delta}\right]^{p/q}dx\right)^{\delta/p}\lesssim \sum_{|n|\ge 2} \frac{ \|f\|_3 |I|^{\tau \delta}}{|n|^{a\delta}}\lesssim  \|f\|_3 |I|^{\tau \delta},
\end{equation*}
is similar.

\subsubsection{Proof of Lemma \ref{lemma:HalfLifting}}\label{Sssec:HalfLifting}

It suffices to consider the cases where $\alpha>0$ is large, and when $\alpha=-1$. The general case then follows from the diagram
\begin{equation*}
\begin{tikzcd}
F_{p,q}^{s,\tau}\left(\D,\XC\right) \arrow{rr}{D^\alpha} \arrow[swap]{dr}{D^{N+\alpha}} & & F_{p,q}^{s-\alpha,\tau}\left(\D,\XC\right) \\
& F_{p,q}^{s-N-\alpha,\tau}\left(\D,\XC\right) \arrow[swap]{ur}{D^{-N}} & 
\end{tikzcd}
\end{equation*}
provided that $N\in\N$ is sufficiently big.

\begin{proof}[The case $\alpha\ge a+2$:]
	We will prove that $ \|D^\alpha f|_{p,q}^{s-\alpha}\|_3\lesssim  \|f|_{p,q}^{s}\|_4$. Let $r_l=e^{-2^{-l}}$. Then $r_l=r_{l+1}^2$. Let $\varphi_{l}^{\left(\alpha\right)}$ be given by $\hat \varphi_{l}^{\left(\alpha\right)}\left(\xi\right)=\left(1+\xi\right)^\alpha e^{-2^{-l}\xi}$ for $\xi\ge -1$. Extend $\hat \varphi_{l}^{\left(\alpha\right)}$ by zero. If $\Phi_l^{\left(\alpha\right)}$ is the corresponding periodization, then $D^\alpha f\left(r_l\zeta_x\right)=\Phi_{l+1}^{\left(\alpha\right)}\ast \Phi_{l+1}\ast f\left(\zeta_x\right)$. This yields that 
	\begin{equation*}
	 \|D^\alpha f\left(r_l\zeta_x\right)\|_\XC\le f_{l+1,a}^*\left(\zeta_x\right)\int _\T |\Phi_{l+1}^{\left(\alpha\right)}\left(\zeta_y\right)|\left(1+2^{l+1}|\zeta_y-1|\right)^a\, dy.
	\end{equation*}
	Similar to Lemma \ref{lemma:CalderonDecay} we have that
	\begin{equation*}
	|\Phi_{l}^{\left(\alpha\right)}\left(\zeta_y\right)|\lesssim \frac{2^{l\left(1+\alpha\right)}}{\left(1+2^l|\zeta_y-1|\right)^N},
	\end{equation*}
	provided that $\alpha >N$. If $\alpha > a +2$, then we may choose $N$ such that $N>a+1$. It follows that
	\begin{equation*}
	 \|D^\alpha f\left(r_l\zeta_x\right)\|_\XC\lesssim 2^{l\alpha} f_{l+1,a}^*\left(\zeta_x\right).
	\end{equation*}
	The statement that $D^\alpha :F_{p,q}^{s,\tau}\left(\D,\XC\right) \to  F_{p,q}^{s-\alpha,\tau}\left(\D,\XC\right) $ now follows from Theorem \ref{thm:Norms}.
\end{proof}

The following estimate follows by the usual tricks, e.g. the proof of \cite{Garnett2007:BddAnalFcnsBook}*{Equation I.3.9}:
\begin{lemma}\label{lemma:RadialGrowth}
	Let $0<p<\infty$. There exists $K=K\left(p\right)>0$ such that,
	\begin{equation*}
	 \|f\left(w\right)\|_\XC\le K \left(1-|w|\right)^{s+\tau-\frac{1}{p}} \|f|_{p,q}^{s,\tau}\|_1
	\end{equation*}
	for $f\in\AC\left(\XC\right)$, $0<q<\infty$, $s,\tau\in\R$. In particular, if $s+\tau-\frac{1}{p}>0$ and $f\in\AC\left(\XC\right)$, then $ \|f|_{p,q}^{s,\tau}\|_1<\infty$ implies that $f\equiv 0$.
\end{lemma}

The following straightforward, yet clever, adaptation of Hardy's inequality is from \cite{Flett1972:DualIneqHardyLittlewood}*{p. 758}:
\begin{lemma}\label{lemma:HardyIneq}
	Let $q\ge 1$ and $\mu<0$. There exists $K=K\left(q,\mu\right)$ such that
	\begin{equation*}
	\int_0^1\left(1-r\right)^{-1-\mu q}\left\{\int_0^r h\left(\rho\right)\, d\rho\right\}^qdr\le K \int_0^1\left(1-r\right)^{-1-\left(\mu-1\right) q}h\left(\rho\right)^q\, dr,
	\end{equation*}
	whenever $h:\left[0,1\right)\to\left[0,\infty\right]$ is measurable.
\end{lemma}

\begin{proof}[The case $\alpha=-1$:]
	The operator $D^{-1}:\AC\left(\XC\right)\to\AC\left(\XC\right)$ has the integral representation
	\begin{equation*}
	\left(D^{-1}f\right)\left(r\zeta_x\right)=\frac{1}{r}\int_0^r f\left(\rho\zeta_x\right)\, d\rho,
	\end{equation*}
	as is verified by term-wise integration of the Taylor series. Let $r_0=1-|I|$, and assume for simplicity that $I\ne \T$, which yields $\frac{1}{r_0}\le 2$. We then have
	\begin{equation*}
	\int_{r=r_0}^1 \left(1-r\right)^{-1-\left(s+1\right)q} \|\left(D^{-1}f\right)_r\left(\zeta_x\right)\|_\XC^q \, dr
	\lesssim 
	A + B,
	\end{equation*}
	where
	\begin{equation*}
	A = \int_{r=r_0}^1 \left(1-r\right)^{-1-\left(s+1\right)q}\left\{\int_{\rho = 0}^{r_0}  \|f_{\rho}\left(\zeta_x\right)\|_\XC\, d\rho \right\}^q \, dr,
	\end{equation*}
	and
	\begin{equation*}
	B = \int_{r=r_0}^1 \left(1-r\right)^{-1-\left(s+1\right)q}\left\{\int_{\rho = r_0}^{1}  \|f_{\rho}\left(\zeta_x\right)\|_\XC\, d\rho \right\}^q \, dr.
	\end{equation*}
	By Lemma \ref{lemma:RadialGrowth}, we trivially obtain the estimate
	\begin{equation*}
	A \lesssim  \|f|_{p,q}^{s,\tau}\|_3^q|I|^{\tau q+q}.
	\end{equation*}
	Applying Lemma \ref{lemma:HardyIneq}, with $h\left(\rho\right)=\1_{\left[r_0,1\right)}\left(\rho\right) \|f_\rho\left(\zeta_x\right)\|_\XC$ and $\mu = s+1$, we obtain
	\begin{equation*}
	B\lesssim \int_{r=r_0}^1\left(1-r\right)^{-1-sq} \|f_r\left(\zeta_x\right)\|_\XC^q\, dr.
	\end{equation*}
	These estimates together show that $D^{-1}:F_{p,q}^{s,\tau}\left(\D,\XC\right)\to F_{p,q}^{s+1,\tau}\left(\D,\XC\right)$ is bounded under the conditions in Theorem \ref{thm:Norms}. This concludes the proof.
\end{proof}

\bibliographystyle{amsplain}
\bibliography{thesisbib}

\end{document}